\documentclass[reqno,12pt]{amsart}
\usepackage{amscd,amssymb,comment,epic,eepic,euscript}
\usepackage{graphicx}
\usepackage{enumerate}
\usepackage[initials]{amsrefs}
\usepackage{array,multicol}
\usepackage{tikz}
\usetikzlibrary{positioning,automata}
\usetikzlibrary{arrows,calc}
\usetikzlibrary{arrows.meta}
\usetikzlibrary{decorations.markings,plotmarks}
\usetikzlibrary{matrix}
\usetikzlibrary{external}
\usepackage{graphics}
\usepackage{color}
\usepackage{amsmath}
\definecolor{red}{rgb}{1,0.00,0.00}

\makeatletter
\newsavebox{\@brx}
\newcommand{\llangle}[1][]{\savebox{\@brx}{\(\m@th{#1\langle}\)}%
  \mathopen{\copy\@brx\mkern2mu\kern-0.9\wd\@brx\usebox{\@brx}}}
\newcommand{\rrangle}[1][]{\savebox{\@brx}{\(\m@th{#1\rangle}\)}%
  \mathclose{\copy\@brx\mkern2mu\kern-0.9\wd\@brx\usebox{\@brx}}}

\makeatother

\author{Ievgen Bondarenko and Kate Juschenko}
\title{\textbf{The zero divisor conjecture and Mealy automata}}

\newcommand\Fb{{\mathbb F}}
\newcommand{\A}{\mathsf{A}}
\newcommand{\B}{\mathsf{B}}
\newcommand{\Helix}{\mathcal{H}}

\newtheorem{theorem}{Theorem}
\newtheorem{proposition}[theorem]{Proposition}

\theoremstyle{definition}
\newtheorem{definition}{Definition}

\newtheorem{conjecture}{Conjecture}
\newtheorem{remark}{Remark}

\begin{document}
\begin{abstract}
The zero divisor conjecture is sufficient to prove for certain class of finitely presented groups where the relations are given by a pairing of generators. We associate Mealy automata to such pairings, and prove that the zero divisor conjecture holds for groups corresponding to invertible automata with three states. In particular, there cannot be zero divisors of support three corresponding to invertible pairings.  
\end{abstract}
\maketitle

In 1956 Irving Kaplansky proposed several problems on group rings. We are interested in two of them: 

\begin{conjecture}[Zero Divisor Conjecture] Let $K$ be a finite field and $G$ be a torsion-free group. Then the group ring $K[G]$ has no zero divisors.
\end{conjecture}

\begin{conjecture}[Direct Finiteness Conjecture] Let $K$ be a finite field and $G$ be a group. Then the group ring $K[G]$ is directly finite, that is, for every $a,b\in K[G]$ the equation $ab=1$ implies $ba=1$.
\end{conjecture}

One of the other famous Kaplansky's conjectures on group rings --- the unit conjecture --- was recently disproved by Giles Gardam in \cite{Gardam}.

Both listed conjectures are open even for the field $\Fb_2$, which we will be considering through the paper. The conjectures were confirmed for many classes of groups. In particular, the Direct Finiteness conjecture is true for sofic groups (see \cite{ES04}). This class of groups is very large, and at the moment there are no examples of non-sofic groups. We refer the reader to a classical book of Passman on group rings \cite{Pa77} for more discussion on Kaplansky's conjectures and their relation to each other.

In \cites{DHJ,DJ} and \cite{Schw}, the authors proposed an algorithmic approach to both conjectures  (the idea circulated among mathematicians for some time, see for example \cite{Thom}). They introduced a class of finitely presented groups $\Gamma_C$ that are universal among possible counterexamples: if $G$ is a counterexample, then $G$ contains a subgroup $H$ that is a quotient of some $\Gamma_C$ that is also a counterexample. Here $C$ is a \textit{pairing matrix}, that is, an integer $n\times m$ matrix $C=(c_{ij})$, where $nm$ is even and $c_{i,j}\in\{1,2,\ldots,nm/2\}$, with the following property: no row or column contain a repeated value, and for every pair $(i,j)$ there exists exactly one pair $(i',j')\neq (i,j)$ such that $c_{ij}=c_{i'j'}$. To every pairing matrix $C$, we can associate a finitely presented group:
\[
\Gamma_C=\langle a_1,\ldots,a_n,b_1,\ldots,b_m\, | \, a_ib_j=a_{i'}b_{j'} \mbox{ whenever $c_{ij}=c_{i'j'}$}\rangle.
\]
Let us recall the connection with the group ring. 

Let $G$ be a group and consider two elements $a,b\in \Fb_2[G]$ satisfying $ab=0$. We may write
\begin{equation*}
a=a_1+\cdots+a_{n} \ \text{ and } \ b=b_1+\cdots+b_{m}
\end{equation*}
for distinct elements $a_1,\ldots,a_{n}$ and distinct elements $b_1,\ldots,b_{m}$ of $G$. In this case we say $a$ has rank $n$ and the support of $a$ is $\{a_1,\ldots,a_n\}$. By distributing the identity $ab=0$, we get $a_ib_j=a_{i'}b_{j'}$ for certain pairing $(i,j)\sim_{\pi} (i',j')$ of the indices. Then the subgroup of $G$ generated by $a_1,\ldots,a_n$ and $b_1,\ldots,b_m$ is a quotient of the group $\Gamma_C$, where the matrix $C$ is given by the pairing $\pi$. Therefore, the Zero Divisor Conjecture over $\Fb_2$ is equivalent to:


\begin{conjecture}
Let $C$ be a pairing matrix and the generators $a_1,\ldots,a_n$ and $b_1,\ldots,b_m$ of $\Gamma_C$ are distinct. Then $\Gamma_C$ contains a nontrivial element of finite order.
\end{conjecture}

%

The groups $\Gamma_C$ were analyzed with respect to these conjectures in \cites{DHJ,Schw}. It was shown that there are no counterexamples with an element of rank two, the groups $\Gamma_C$ for $2\times n$ pairings satisfy the conjectures. With the computer assistance, all groups $\Gamma_C$ up to sizes $3\times 11$ and $5\times 5$ were determined in \cite{DHJ}, and they satisfy the Direct Finiteness Conjecture. In \cite{Schw}, it is shown that there are no zero divisors, where the ranks of elements are $(3,m)$ for $m\leq 15$ and $(4,m)$ for $m\leq 6$. In \cites{Ali,Ali2}, this result was extended to the ranks $(3,m)$ for $m\leq 19$ and $(4,m)$ for $m\leq 8$.

In this paper, we associate Mealy automata to pairings matrices and relate certain cycles in these automata to elements of finite order. We prove that the group $\Gamma_C$ corresponding to invertible automata with three states contains nontrivial elements of finite order (Theorem~4). In particular, there cannot be zero divisors of support three corresponding to invertible pairings.

\section{Mealy automata and pairing matrices}

A Mealy automaton over an alphabet $X$ is a quadruple $\A = (S, X, \mu, \nu)$, where $S$ is the set of states; $X$ is an alphabet; $\mu : S \times X \rightarrow S$ is the transition map; and $\nu  : S \times X \rightarrow X$ is the output map. An automaton $\A$ can be identified with a directed labeled graph whose vertices are identified with the states of $\A$, and for every state $s \in S$ and every letter $x \in X$, the graph has an arrow from $s$ to $\mu(s, x)$ labeled by $x|\nu(s, x)$.


Let $C$ be a $n\times m$ pairing matrix. We associate to $C$ a Mealy automaton $\A_C$ with the set of states $S=\{a_1,\ldots,a_n\}$ over the alphabet $X=\{b_1,\ldots,b_m\}$ and the following arrows:
\[
a_i\xrightarrow{b_j|b_{j'}}a_{i'} \ \mbox{ whenever $c_{ij}=c_{i'j'}$},
\]
see example in Figure~\ref{fig_automaton_Ac}.

\vspace{-0.5cm}
\begin{figure}[h]
\begin{multicols}{2}
\[
C=\left(
    \begin{array}{cccc}
      1 & 2 & 3 & 4 \\
      5 & 6 & 2 & 1 \\
      3 & 4 & 5 & 6 \\
    \end{array}
  \right)
\]\vspace{1cm}

\begin{tikzpicture}[>=stealth, scale=0.6, shorten >=2pt, on grid, auto,thick,every initial
by arrow/.style={*->}, el/.style = {inner sep=2pt, align=left, sloped},
every label/.append style = {font=\tiny}]
   \node[state] at (0,0) (a)   {$a_1$};
   \node[state] at (-3,-5) (b) {$a_2$};
   \node[state] at (3,-5) (c)  {$a_3$};

   \tikzstyle{every node}=[font=\footnotesize]
    \path[->]
    (a) edge [bend right] node [el] {$b_1|b_4$ \\ $b_2|b_3$} (b)
    (b) edge [bend left=20] node [el,below] {$b_4|b_1$ \\ $b_3|b_1$} (a)    
    (b) edge [bend right] node [el,below] {$b_1|b_3$, $b_2|b_4$} (c)
    (c) edge [bend left=20] node [el,above] {$b_3|b_1$, $b_4|b_2$} (b)
    (c) edge [bend right] node [el,above] {$b_1|b_3$ \\ $b_2|b_4$} (a)
    (a) edge [bend left=20] node [el,below] {$b_3|b_1$ \\ $b_4|b_2$} (c);
\end{tikzpicture}
\end{multicols}
\caption{Example of a pairing matrix $C$ and the associated automaton $\A_C$}
\label{fig_automaton_Ac}
\end{figure}
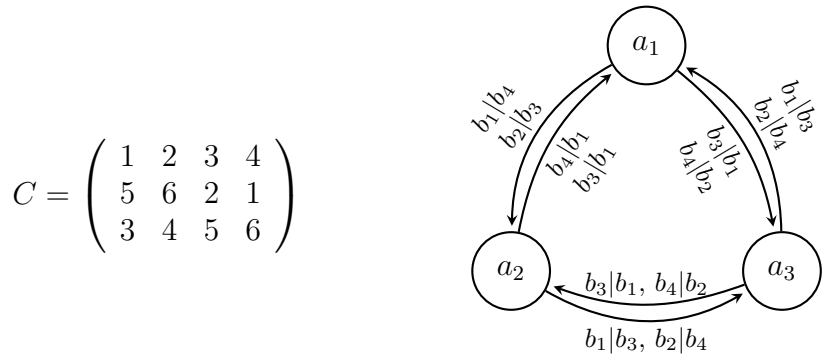

An automaton $\A=(S, X, \mu, \nu)$ is called invertible if, for all $s \in S$, the transformation $\nu(s,\cdot): X \rightarrow X$ is a permutation of $X$. For invertible automata, one can naturally extend the functions $\mu,\nu$ to $S^{-1}\times X$ as $\mu(s^{-1},y)=t^{-1}$ and $\nu(s^{-1},y)=x$, where $y=\nu(s,x)$ and $t=\mu(s,x)$. An invertible automaton $\A$ is called bireversible if the map $$(s,x)\mapsto (\mu(s,x),\nu(s,x))$$ is a bijection on $S\times X$ and $S^{-1}\times X$. For more information about Mealy automata and related groups, see \cites{BK,GNS}.

\begin{proposition}\label{prop_AC_invertible}
Let $C$ be a pairing matrix and $\A_C$ the associated automaton. The following statements are equivalent:
\begin{enumerate}
  \item every row and every column of $C$ have exactly one equal coefficient besides the common coefficient in their intersection;
  \item $\A_C$ is invertible;
  \item $\A_C$ is bireversible.
\end{enumerate}
\end{proposition}
\begin{proof}
The statement in the first item is equivalent to saying that for every row $i$ the map $j\mapsto j'$ defined on $\{1,2,\ldots, m\}$ by the relation $c_{ij}=c_{i'j'}$ is a permutation. This exactly means that the automaton $\A_C$ is invertible. Another reformulation of invertability: for every pair $(i,j')$ there exists a unique pair $(i',j)$ such that $c_{ij}=c_{i'j'}$. It follows that for every $(i',j')$ there exists a unique $(i,j)$ such that $c_{i'j'}=c_{ij}$, and for every $(i',j)$ there exists a unique $(i,j')$ with $c_{i'j'}=c_{ij}$. This means that $\A_C$ is bireversible.
\end{proof}

\begin{definition}
We say that a matrix $C$ defines an invertible pairing if it satisfies the conditions of Proposition~\ref{prop_AC_invertible}.
\end{definition}

We can characterize invertible automata corresponding to pairings in term of their Helix graph. The Helix graph of an automaton $\A$ is the directed graph $\Helix_{\A}$ on the vertex set $S\times X$ with arrows
\[
(s,x)\rightarrow (t,y) \mbox{ in $\Helix_{\A}$ \quad whenever \quad } s\xrightarrow{x|y}t \mbox{ in $\A$}.
\]

\vspace{0.0cm}
\begin{proposition}
A bireversible automaton $\A$ corresponds to a certain pairing matrix if and only if $\A$ does not contain loops and arrows of the form $s\xrightarrow{x|x} t$, and the Helix graph of $\A$ consists of cycles of length two.
\end{proposition}
\begin{proof}
The Helix graph of a bireversible automaton consists of cycles. If $\A$ corresponds to a paring matrix $C$, then the pairing $c_{ij}=c_{i'j'}$ implies a cycle in the Helix graph $\Helix_\A$:
\[
(a_i,b_j)\rightarrow (a_{i'},b_{j'})\rightarrow (a_i,b_j).
\]
The absence of loops and arrows $s\xrightarrow{x|x} t$ is equivalent to absence of repeated values in rows and columns of $C$. Conversely, if all cycles have length two, then we can construct a matrix $C=(c_{sx})$ with rows labeled by the states and columns labeled by the letters such that
\[
c_{sx}=c_{ty} \ \mbox{ if and only if \  $(s,x)\rightarrow (t,y)\rightarrow (s,x)$ in $\Helix_\A$}.
\]
The matrix $C$ is a pairing matrix, because $(t,y)\neq (s,x)$ and $(t,y)$ is uniquely determined by $(s,x)$. 
\end{proof}

An automaton $\A=(S,X,\mu,\nu)$ can be uniquely extended to an automaton $\A^{*}$, where the states and letters are the words over $S$ and $X$ respectively:
\begin{align*}
s\xrightarrow{x_1|y_1}q \mbox{ \ and \ } q\xrightarrow{x_2|y_2}t \quad &\Rightarrow \quad s\xrightarrow{x_1x_2|y_1y_2}t,\\
s_1\xrightarrow{x|z}t_1 \mbox{ \ and \ } s_2\xrightarrow{z|y}t_2 \quad &\Rightarrow \quad s_1s_2\xrightarrow{x|y}t_1t_2.
\end{align*}
Certain cycles in the automaton $\A_C^{*}$ produce elements of finite order in the group $\Gamma_C$.

\begin{proposition}
Let $C$ be a pairing matrix. Then
\begin{eqnarray*}
s\xrightarrow{xx\ldots x|yy\ldots y}s \mbox{\quad in \ $\A_C^{*}$} \ &\Rightarrow& \ (yx^{-1})^k=e \mbox{\quad in \ $\Gamma_C$},\\
ss\ldots s\xrightarrow{x|x}tt\ldots t \mbox{\quad in \ $\A_C^{*}$} \ &\Rightarrow& \ (t^{-1}s)^k=e \mbox{\quad in \ $\Gamma_C$}.
\end{eqnarray*}
\end{proposition}
\begin{proof}
The arrow $s\xrightarrow{xx\ldots x|yy\ldots y}s$ in $\A_C^{*}$ means that we have the following sequence of arrows in $\A_C$:
\[
s\xrightarrow{x|y}s_2, \ s_2\xrightarrow{x|y}s_3, \ \ldots, \ s_k\xrightarrow{x|y}s.
\]
These arrows correspond to the following relations in $\Gamma_C$:
\[
sx=s_2y, \ s_2x=s_3y, \ \ldots \ , s_kx=sy.
\]
Then
\begin{gather*}
sx=s_2y=s_3yx^{-1}y=s_4yx^{-1}yx^{-1}y=\ldots=s(yx^{-1})^{k-1}y\\  \Rightarrow \quad (yx^{-1})^k=e.
\end{gather*}
\end{proof}


The Kaplansky's zero divisor conjecture can be equivalently formulated as $a^2=0$ implies $a=0$ instead of $ab=0$ implies $a=0$ or $b=0$. This allows to reduce the generators of the associated finitely presented group to $a_1,\ldots,a_n$. Let $C$ be an $n\times n$ pairing matrix (here $n$ is even), and define a finitely presented group
\[
G_C=\langle a_1,\ldots,a_n\ |\ a_ia_j=a_{i'}a_{j'} \mbox{ whenever } c_{ij}=c_{i'j'}\rangle.
\]
Then the Kaplansky's conjecture is equivalent to the existence of non-trivial elements of finite order in the group $G_C$.  
We get a different way to associate an automaton to a pairing matrix: define an automaton $\B_C$ with the set of states $S=\{a_1,\ldots,a_n\}$, alphabet $X=\{a_1,\ldots,a_n\}$, and arrows
\[
a_i\xrightarrow{a_j|a_{i'}}a_{j'} \ \mbox{ whenever $c_{ij}=c_{i'j'}$}.
\]
Then any transition $s\xrightarrow{v | u} t$ in $\B_C^{*}$ implies the relation $sv=ut$ in $G_C$.
Unfortunately, the automaton $\B_C$ is never invertible. Indeed, if the first row of $C$ is $[1,2,\ldots,n]$, then some other row $k$ contains at least two numbers $1\leq i<j\leq n$. This means that there are two arrows in $\B_C$:
\[
a_1\xrightarrow{a_i|a_k} b\qquad \mbox{ and } \qquad a_1\xrightarrow{a_j|a_k} b',
\]
what contradicts invertibility.

\section{The zero divisor conjecture for invertible pairings}

We show that the zero divisor conjecture holds for $3\times 2n$ invertible pairings.

\begin{theorem}
Let $C$ be a $3\times 2n$ invertible pairing matrix. Let us assume that the generators $a_1$, $a_2$, $a_3$ and $b_1$, \ldots, $b_{2n}$ of the group $\Gamma_{C}$ are distinct. Then $\Gamma_C$ contains a nontrivial element of finite order.
\end{theorem}
\begin{proof}
Each row and column of $C$ do not contain two equal coefficients, since otherwise $a_i=a_j$ or $b_i=b_j$ for $i\neq j$.
We want to show that by permuting the columns/rows and relabeling equal coefficients, the matrix $C$ can be brought to the following form:
\[
C=\left(
    \begin{array}{cccccccc}
      1 & 2 & \ldots & n & n+1 & n+2 & \ldots & 2n \\
2n+1 & 2n+2 & \ldots & 3n & c_{2,n+1} & c_{2,n+2} & \ldots & c_{2,2n} \\
      c_{3,1} & c_{3,2} & \ldots & c_{3,n} & c_{3,n+1} & c_{3,n+2} & \ldots & c_{3,2n} \\
    \end{array}
  \right),
\]
where $(c_{2,n+1}, c_{2,n+2}, \ldots, c_{2,2n})$ is a permutation of $1,2,\ldots,n$, $(c_{3,1}, c_{3,2}, \ldots, c_{3,n})$ is a permutation of $n+1, n+2, \ldots, 2n$, and $(c_{3,n+1}, c_{3,n+2}, \ldots, c_{3,2n})$ is a permutation of $2n+1, 2n+2, \ldots, 3n$. This gives certain ``bipartite'' structure on $C$.

First, half of the numbers $1,2,\ldots,2n$ should be in the second row and the other half in the third one, since otherwise there would be two equal elements in one row. By permuting the rows if necessary, we can assume that $1,2,\ldots,n$ are in the second row.

Second, the numbers $1,2,\ldots,n$ in the second row cannot belong to columns $1$, $2$,\ldots, $n$. Indeed, if $c_{2,i}\in\{1,2,\ldots,n\}$ for $1\leq i\leq n$, then the $i$th column and $3$rd row do not contain common coefficient except for their intersection; this contradicts invertibility. Therefore, $(c_{2,n+1}, c_{2,n+2}, \ldots, c_{2,2n})$ is a permutation of $1,2,\ldots,n$. By the same reason, $(c_{3,1}, c_{3,2}, \ldots, c_{3,n})$ is a permutation of $n+1, n+2, \ldots, 2n$ and $(c_{3,n+1}, c_{3,n+2}, \ldots, c_{3,2n})$ is a permutation of $2n+1, 2n+2, \ldots, 3n$.

Let us denote $x_i=b_i$ and $y_i=b_{n+i}$ for $i=1,\ldots,n$ in order to reflect the bipartite structure of $C$. Then we get the following relations in $\Gamma_C$:
\begin{align*}
t_1=a_1^{-1}a_2&=x_1y_{i_1}^{-1}=x_2y_{i_2}^{-1}=\ldots=x_ny_{i_n}^{-1},\\
t_2=a_2^{-1}a_3&=x_1y_{j_1}^{-1}=x_2y_{j_2}^{-1}=\ldots=x_ny_{j_n}^{-1},\\
t_3=a_3^{-1}a_1&=x_1y_{k_1}^{-1}=x_2y_{k_2}^{-1}=\ldots=x_ny_{k_n}^{-1},
\end{align*}
where $(i_1,i_2,\ldots,i_n)$, $(j_1,j_2,\ldots,j_n)$, and $(k_1,k_2,\ldots,k_n)$ are permutations of $1,2,\ldots,n$.
Notice that if we write the corresponding relations for each $t_1^{-1}, t_2^{-1}, t_3^{-1}$, we will see $y_1x_{l_1}^{-1}, y_2x_{l_2}^{-1},\ldots, y_nx_{l_n}^{-1}$, where $(l_1,l_2,\ldots,l_n)$ is a permutation of $1,2,\ldots,n$.

We want to show that $t_1t_2^{-1}$ has finite order. Let us construct a directed $(n,n)$ bipartite graph: the left part consists of vertices $x_1y_{i_1}^{-1}$, \ldots, $x_ny_{i_n}^{-1}$, which represent $t_1$, the right part consists of vertices $y_1x_{l_1}^{-1}, y_2x_{l_2}^{-1},\ldots, y_nx_{l_n}^{-1}$, which represent $t_2^{-1}$. We put a directed edge from $x_sy_{i_s}^{-1}$ to $y_{i_s}x_{l_{is}}^{-1}$ for $s=1,\ldots,n$, and a directed edge from $y_{s}x_{l_s}^{-1}$ to $x_{l_s}y^{-1}_{i_{l_s}}$ for $s=1,\ldots,n$. Now there is a directed edge passing from any vertex. Hence, there is a directed cycle of some length $1\leq 2m_{12}\leq 2n$. By multiplying consequently the elements along this cycle, we get  $(t_1t_2^{-1})^{m_{12}}=e$.

The same observation holds for $t_2t_3^{-1}$ and $t_3t_1^{-1}$. Hence
\[
(t_1t_2^{-1})^{m_{12}}=(t_2t_3^{-1})^{m_{23}}=(t_3t_1^{-1})^{m_{31}}=e
\]
for some $1\leq m_{ij}\leq n$. If $t_i\neq t_j$ for some $i\neq j$, we get a nontrivial element $t_it_j^{-1}$ of finite order. If  $t_1=t_2=t_3$, then $$t_1^3=t_1t_2t_3=a_1^{-1}a_2\cdot a_2^{-1}a_3\cdot a_3^{-1}a_1=e$$ and $t_1=a_1^{-1}a_2\neq e$ by assumption. The theorem is proved.
\end{proof}

\begin{remark}
The previous proof also works for larger pairing matrices $C$/automata $\A_C$ as soon as they admit the following cyclic-bipartite structure: there is an order on the states $s_1,s_2,\ldots,s_n$ and a partition of the alphabet $X=A\sqcup B$, $|A|=|B|$ so that the arrows in $\A_C$ are of the form
\[
s_i\xrightarrow{A|B} s_{i+1} \ \mbox{ and } \ s_i\xrightarrow{B|A} s_{i-1}.
\]
Invertible pairing $3\times 2n$ matrices always admit such a structure.
\end{remark}

\begin{remark}
The automata $\A_C$ associated to invertible pairings are bireversible. There is a strong connection between bireversible automata and lattices in the product of the automorphism groups of two regular trees that are given by presentation
\[
G_{\A}=\{S,X\, |\, sx=yt \mbox{ whenever $s\xrightarrow{x|y}t$ in $\A$ } \}.
\]
The groups $G_{\A}$ are always torsion-free; however, some of them are non-residually finite, even virtually simple, and not elementary sofic (see \cite{BK} are the references therein). It is unknown whether the groups $G_{\A}$ are sofic or satisfy the Kaplansky's conjectures.
\end{remark}


\begin{bibdiv}
\begin{biblist}

\bib{Ali}{article}{
author = {Alireza Abdollahi and Zahra Taheri},
title = {Zero divisors and units with small supports in group algebras of torsion-free groups},
journal = {Communications in Algebra},
volume = {46},
number = {2},
pages = {887--925},
year = {2018},
publisher = {Taylor & Francis},
doi = {10.1080/00927872.2017.1344688}
}

\bib{Ali2}{article}{
author = {Alireza Abdollahi and Fatemeh Jafari},
title = {Zero divisor and unit elements with supports of size $4$ in group algebras of torsion-free groups},
journal = {Communications in Algebra},
volume = {47},
number = {1},
pages = {424-449},
year = {2019},
publisher = {Taylor & Francis},
doi = {10.1080/00927872.2018.1477949}
}

\bib{BK}{article}{
author = {Bondarenko, Ievgen},
author = {Kivva, Bohdan},
title = {Automaton groups and complete square complexes},
journal = {Groups, Geometry, and Dynamics},
volume = {16},
number = {1},
pages = {305--332},
year = {2022}
}



\bib{DHJ}{article}{
  author={Dykema, Kenneth},
  author={Heister, Timothy},
  author={Juschenko, Kate}
  title={Finitely presented groups related to Kaplansky's direct finiteness conjecture},
  journal={Experimental Mathematics},
  year={2015},
  volume={24(3)},
  pages={326-338}
}

\bib{DJ}{article}{
  author={Dykema, Kenneth},
  author={Juschenko, Kate}
  title={On stable finiteness of group rings},
  journal={Algebra and Discrete Mathematics},
  year={2018},
  volume={19(1)},
}

\bib{ES04}{article}{
  author={Elek, G\'abor},
  author={Szab\'o, Endre},
  title={Sofic groups and direct finiteness},
  journal={J. Algebra},
  year={2004},
  volume={280},
  pages={426--434}
}

\bib{Gardam}{article}{
 Author = {Gardam, Giles},
 Title = {A counterexample to the unit conjecture for group rings},
 FJournal = {Annals of Mathematics. Second Series},
 Journal = {Ann. Math. (2)},
 ISSN = {0003-486X},
 Volume = {194},
 Number = {3},
 Pages = {967--979},
 Year = {2021},
 DOI = {10.4007/annals.2021.194.3.9}
}

\bib{GNS}{article}{
 author = {R.I. Grigorchuk, V.V. Nekrashevych, V.I.~Sushchansky},
 title =  {Automata, dynamical systems and groups},
 journal = {Proceedings of the Steklov Institute of Mathematics},
 volume = {231},
 pages = {128--203},
 year = {2000}
}

\bib{Pa77}{book}{
   author={Passman, Donald S.},
   title={The algebraic structure of group rings},
   series={Pure and Applied Mathematics},
   publisher={Wiley-Interscience [John Wiley \& Sons]},
   place={New York},
   date={1977}
}


\bib{Schw}{article}{
  author={Schweitzer, Pascal},
  title={On zero divisors with small support in group rings of torsion--free groups},
  journal={Journal of Group Theory},
  volume={16},
  number={5},
  year={2013},
  pages={667--693}  
}

\bib{Thom}{misc}{
    title={Zero divisor conjecture and idempotent conjecture},
    author={Andreas Thom\phantom{x}(mathoverflow.net/users/8176)},
    note={\url{http://mathoverflow.net/questions/34616} (version: 2010-08-05)},
    eprint={http://mathoverflow.net/questions/34616},
    organization={MathOverflow}, 
}



























\end{biblist}
\end{bibdiv}
\end{document}